\documentclass[10pt]{amsart}
\usepackage[utf8]{inputenc}
\usepackage[T1]{fontenc}
\usepackage{geometry}
\usepackage[english]{babel}
\usepackage[centertags]{amsmath}
\usepackage{amstext,amssymb,amsopn,amsthm}
\usepackage{nicefrac, esint}
\usepackage{mathrsfs}
\usepackage{dsfont}
\usepackage{bbm}
\usepackage{thmtools}
\usepackage{graphicx}
\usepackage[titletoc,title]{appendix}

\usepackage[backgroundcolor=white, bordercolor=blue,
linecolor=blue]{todonotes}
\usepackage[colorlinks=true, linkcolor=black, citecolor=black]{hyperref}
\usetikzlibrary{patterns}
\usetikzlibrary{arrows}
\usepackage{enumitem}
\setlist[enumerate]{itemsep=0mm}
\usepackage{bookmark}
\addto\extrasenglish{}
\addto\extrasenglish{}
\addto\extrasenglish{}

\declaretheorem[name=Theorem, numberwithin=section]{theorem}
\newtheorem{corollary}[theorem]{Corollary}
\newtheorem{lemma}[theorem]{Lemma}
\newtheorem{proposition}[theorem]{Proposition}
\newtheorem{definition}[theorem]{Definition}

\newtheorem*{example*}{Example}
\newtheorem*{assumption}{Assumption}
\newtheorem*{theorem*}{Theorem}
\newtheorem*{lemma*}{Lemma}

\declaretheoremstyle[bodyfont=\normalfont]{remark-style}

\numberwithin{equation}{section}

\theoremstyle{plain}

\newcommand{\N}{\mathds{N}}
\newcommand{\R}{\mathds{R}}

\newcommand{\m}{\varpi}

\newcommand{\BIGOP}[1]
{
\mathop{\mathchoice%
{\raise-0.22em\hbox{\huge $#1$}}%
{\raise-0.05em\hbox{\Large $#1$}}{\hbox{\large $#1$}}{#1}}}
\newcommand{\bigtimes}{\BIGOP{\times}}

\def\XXint#1#2#3{{\setbox0=\hbox{$#1{#2#3}{\int}$}
     \vcenter{\hbox{$#2#3$}}\kern-.5\wd0}}

\newcommand{\BIGboxplus}{\mathop{\mathchoice%
{\raise-0.35em\hbox{\huge $\boxplus$}}%
{\raise-0.15em\hbox{\Large $\boxplus$}}{\hbox{\large $\boxplus$}}{\boxplus}}}

\DeclareMathOperator{\dist}{dist}

\DeclareMathOperator{\dvg}{div}

\DeclareMathOperator*{\osc}{osc}

\renewcommand{\d}{\textnormal{d}}

\allowdisplaybreaks
\begin{document}

\title{Regularity of solutions to anisotropic nonlocal equations}
\author{Jamil Chaker}
\email{jchaker@math.uni-bielefeld.de}
\address{Universit\"{a}t Bielefeld, Fakult\"{a}t f\"{u}r Mathematik, Postfach
100131, D-33501 Bielefeld, Germany}
\begin{abstract} We study harmonic functions associated to systems of stochastic differential equations of the form $dX_t^i=A_{i1}(X_{t-})dZ_t^1+\cdots+A_{id}(X_{t-})dZ_t^d$, $i\in\{1,\dots,d\}$,
where $Z_t^j$ are independent one-dimensional symmetric stable processes of order $\alpha_j\in(0,2)$, $j\in\{1,\dots,d\}$. 
In this article we prove H\"older regularity of bounded harmonic functions with respect to solutions to such systems. 
\end{abstract}
\maketitle
\begin{small}\textbf{AMS 2010 Mathematics Subject Classification: } Primary 60J75; Secondary  60H10, 31B05, 60G52	                                                                                                       \end{small}

\begin{small}\textbf{Keywords:} Jump processes, Harmonic functions, H\"older continuity, Support theorem, Anisotropy, Nonlocal Operators\end{small}

\section{Introduction}
The consideration of stochastic processes with jumps and anisotropic behavior is natural and reasonable since such objects arise in several natural and financial models. 
In certain circumstances L\'evy processes with jumps are more suitable to capture empirical facts that diffusion models do. See for instance \cite{CTF} for examples of financial models with jumps. 

In the nineteen fifties, De Giorgi \cite{DEGIORGI} and Nash \cite{NASH} independently prove an a-priori H\"older estimate for weak solutions $u$ to second order equations of the form
\[ \dvg(A(x)\nabla u(x))=0 \]
for uniformly elliptic and measurable coefficients $A$. 
In \cite{MOSER}, Moser proves H\"older continuity of weak solutions and gives a proof of an elliptic Harnack inequality for weak solutions
to this equation. This article provides a new technique of how to derive an a-priori H\"older estimate from the Harnack inequality. For a large class of local operators, the 
H\"older continuity can be derived from the Harnack inequality, see for instance \cite{GILBARG}.
For a comprehensive introduction into Harnack inequalities, we refer the reader e.g. to \cite{INTROHAR}.

The corresponding case of operators in non-divergence form is treated in by Krylov and Safonov in \cite{KSA}. The authors develop a technique for proving H\"older regularity and the Harnack inequality for harmonic functions 
corresponding to non-divergence form elliptic operators. They take a probabilistic point of view and make use of the martingale problem to prove regularity estimates for harmonic functions. 
The main tool is a support theorem, which gives information about the topological support for solutions to the martingale problem associated to the corresponding operator. 
This technique is also used in \cite{BLH} to prove similar results for nonlocal operators of the form 
\begin{equation}\label{eq:integrodiffa}
Lf(x)=\int_{\R^d\setminus\{0\}}[f(x+h)-f(x)-\mathds{1}_{\{|h|\leq1\}}h\cdot \nabla f(x)]a(x,h)dh
\end{equation}
under suitable assumptions on the function $a$. In \cite{BCR} Bass and Chen follow the same ideas to prove H\"older
regularity for harmonic functions associated to solutions of systems of stochastic differential equations driven by L\'{e}vy processes with highly singular L\'{e}vy measures.
In this work we extend the results obtained by Bass and Chen to a larger class of driving L\'{e}vy processes.

A one-dimensional L\'{e}vy process $(Y_t)_{t\geq 0}$ is called \textit{symmetric stable processes of order} $\mathit{\gamma \in (0,2)}$ if its
characteristic function is given by
\[ \mathbb{E}e^{i\xi Y_t} = e^{-t|\xi|^\gamma}, \qquad \xi\in\R.\]
The L\'{e}vy measure of such a process is given by $\nu(dh)=c_{\gamma}|h|^{-1-\gamma}\, dh,$ where $c_{\gamma}=2^{\gamma}\Gamma\left(\frac{1+\gamma}{2}\right)/\left|\Gamma\left(-\frac{\gamma}{2}\right)\right|$.

Let $d\in\N$  and $d\geq 2$. We assume that $Z_{t}^{i}, i=1,\dots,d$, are independent one-dimensional symmetric stable processes of order $\alpha_{i} \in (0,2)$
and define $Z=(Z_t)_{t\geq0}=(Z_t^{1},\dots,Z_t^{d})_{t\geq 0}$.\\ 
The L\'{e}vy-measure of this process is supported on the coordinate axes and is given by
\[  \nu(dw)= \sum_{k=1}^{d} \left( \frac{c_{\alpha_k}}{|w_k|^{1+\alpha_k}}dw_k \left(\prod_{j\neq k}\delta_{\{0\}}(dw_j) \right)\right).\]
Therefore $\nu(A)=0$ for every set $A\subset \R^d$, which has an empty intersection with the coordinate axes.
The generator $L$ of $Z$ is given for $f\in C_b^2(\R^d)$ by the formula
\begin{equation}\label{eq:generator}
 Lf(x)=\sum_{k=1}^{d} \int_{\R\setminus \{0\}} (f(x+he_{k})-f(x)-\mathds{1}_{\{|h|\leq 1\}}\partial_kf(x)h)\frac{c_{{\alpha_k}}}{|h|^{1+\alpha_k}} dh. 
 \end{equation}
For a deeper discussion on L\'{e}vy processes and their generators we refer the reader to \cite{SAT}.

Let $x_0\in\R^d$ and $A:\R^d\to\R^{d\times d}$ a matrix-valued function. We consider the system of stochastic differential equations
\begin{equation}\label{system of equations}
\left\{\begin{aligned}
  dX^i_{t} &= \sum_{j=1}^d A_{ij}(X_{t-})dZ_{t}^j, \\
   X_0^i &= x_0^i,
\end{aligned} \right.
\end{equation}
where $X_{t-}=\lim\limits_{s\nearrow t} X_s$ is the left hand limit.\\
  This system has been studied systematically in the case $\alpha_1=\alpha_2=\cdots=\alpha_d=\alpha\in(0,2)$ by Bass and Chen in the articles \cite{BCS} and \cite{BCR}. 
 With the help of the martingale problem, Bass and Chen prove in \cite{BCS} that for each $x_0\in\R^d$
 there exists a unique weak solution $(X=(X_t^1,\dots,X_t^d)_{t\geq 0}, \mathbb{P}^{x_0})$ to \eqref{system of equations}. 
 Furthermore the authors prove that the family $\{X,\mathbb{P}^x,x\in\R^d\}$ forms a conservative strong Markov process on $\R^d$ whose semigroup
 maps bounded continuous functions to bounded continuous functions (see Theorem 1.1, \cite{BCS}). 
 Consequently it follows that
\[  \mathcal{L}f(x)=\sum_{j=1}^{d} \int_{\R\setminus \{0\}} (f(x+a_{j}(x)h)-f(x)-h\mathds{1}_{\{|h|\leq 1\}}\nabla f(x) \cdot a_j(x))\frac{c_{\alpha}}{|h|^{1+\alpha}} dh \]
  coincides on $C_b^2(\R^d)$ with the generator for any weak solution to \eqref{system of equations},  where $a_j(x)$ denotes the j$^{th}$ column of the matrix $A(x)$.
In \cite{BCR} the authors prove H\"older regularity of harmonic functions with respect to $\mathcal{L}$ and give a counter example which shows that the Harnack inequality for harmonic functions is not satisfied.

 In this paper we do not study unique solvability of \eqref{system of equations} but prove an a-priori regularity estimate for harmonic functions if unique solutions to the system exist. The following assumptions will be needed throughout the paper.
  \begin{assumption}{\ } \vspace*{-1.5em}
  \begin{enumerate}
   \item[(i)] For every $x\in\R^d$ the matrix $A(x)$ is non-degenerate, that is $\det (A(x))\neq 0$.
   \item[(ii)] The functions $x\mapsto A_{ij}(x)$ and $x\mapsto A^{-1}_{ij}(x)$ are continuous and bounded for all $1\leq i,j\leq d$ and $x\in\R^d$.  
   \item[(iii)] \label{martingale-problem} For any $x_0\in\R^d$, there exists a unique solution to the martingale problem for 
\begin{equation}\label{gen:op}   
\mathcal{L}f(x)=\sum_{j=1}^{d} \int_{\R\setminus \{0\}} (f(x+a_{j}(x)h)-f(x)-h\mathds{1}_{\{|h|\leq 1\}}\nabla f(x) \cdot a_j(x))\frac{c_{\alpha_j}}{|h|^{1+\alpha_j}} dh
\end{equation}
started at $x_0$. The operator $\mathcal{L}$ coincides on $C_b^2(\R^d)$ with the generator for the weak solution to \eqref{system of equations}.
  \end{enumerate}
  \end{assumption}
 For a comprehensive introduction into the martingale problem we refer the reader to \cite{ETHIER}. 
\subsection*{Notation} Let $A$ be the matrix-valued function from \eqref{system of equations}. Let $D$ be a Borel set. 
Throughout the paper $\m(D)$ denotes the modulus of continuity of $A$ and we write $\Lambda(D)$ for the upper bound of $A$ on $D$. 
We set $\alpha_{\min}:=\min\{\alpha_1,\dots,\alpha_d\}$ and $\alpha_{\max}:=\max\{\alpha_1,\dots,\alpha_d\}$.
For $i\in\N$ we write $c_i$ for positive constants and additionally $c_i=c_i(\cdot)$ if we want to highlight all the quantities the constant depends on.

In order to deal with the anisotropy of the process we consider a corresponding scale of cubes.
\begin{definition}\label{rectangles}
 Let $r\in(0,1]$ and $\alpha_1,\dots,\alpha_d\in(0,2)$.
 For $k>0$, we define
 \[ M^{k}_r(x):= \bigtimes_{i=1}^d \left(x_i-(kr^{\alpha_{\max}/\alpha_i}),x_i+(kr^{\alpha_{\max}/\alpha_i})\right). \]
For brevity we write $M_r(x)$ instead of $M^{1}_r(x)$.
 \end{definition}
Note that $M_r^k$ is increasing in $k$ and $r$. 
For $z\in\R^d$ and $r\in(0,1]$, the set $M_r(z)$ is a ball with radius $r$ and center $z$ in the metric space $(\R^d,d)$, where 
\[ d(x,y)=\sup_{k\in\{1,\dots,d\}} \{|x_k-y_k|^{\alpha_k/\alpha_{\max}}\mathds{1}_{\{|x_k-y_k|\leq 1\}}(x,y) + \mathds{1}_{\{|x_k-y_k|> 1\}}(x,y)\}. \]
This metric is useful for local considerations only, that is studies of balls with radii less or equal than one.
The advantage of using these sets is the fact that they reflect the different jump intensities of the process $Z$
and compensate them in an appropriate way, see for instance \autoref{exittime}.

The purpose of this paper is to prove the following result.
\begin{theorem}\label{hoelder-reg}
 Let $r\in(0,1]$, $s>0$ and $x_0\in\R^d.$ Suppose $h$ is bounded in $\R^d$ and harmonic in $M_r^{1+s}(x_0)$ 
 with respect to $X$.  
 Then there exist $c_1=c_1(\Lambda(M_r^{1+s}(x_0)),\m(M_r^{1+s}(x_0)))>0$ and $\beta=\beta(\Lambda(M_r^{1+s}(x_0)),\m(M_r^{1+s}(x_0)))>0$,
 independent of $h$ and $r$, such that
 \[ |h(x)-h(y)|\leq c_1 \left( \frac{|x-y|}{r^{\alpha_{\max}/\alpha_{\min}}} \right)^{\beta} \sup\limits_{\R^d} |h(z)|  \quad \text{ for } x,y\in M_r(x_0).\]
 \end{theorem}

We want to emphasize, that in the case $\alpha_1=\cdots=\alpha_d$ the set $M_r(x_0)$ reduces to a cube with radius $r$
and hence this result coincides with \cite[Theorem 2.9]{BCR}, when one chooses cubes instead of balls.

Let us briefly discuss selected related results in the literature.

As previously mentioned, in \cite{BLH} the authors study operators of the form \eqref{eq:integrodiffa} for coefficients $a:\R^d\times\R^d\to\R$ which are assumed to be symmetric in the second variable and satisfy $a(x,h)\asymp |h|^{-d-\alpha}$ for all $x,h\in\R^d$, where $\alpha\in(0,2)$. 
Using probabilistic techniques they prove a Harnack inequality and derive H\"older regularity estimates for bounded harmonic functions.
The results of this work have been extended to more general kernels by several authors. 
For instance, in \cite{BASKAS1} the authors establish a H\"older estimate for harmonic functions to operators of the form \eqref{eq:integrodiffa}, where they replace the jump measure $a(x,h)\, \d h$ by a family of measures $n(x,\d h)$, which is not required to have a density with respect to the Lebesgue meaure. 
Furthermore, \cite{SONG04} extends the method of \cite{BLH} to prove the Harnack inequality for more general classes of Markov processes.
In \cite{Bo17} the authors construct and study the heat kernel a class of highly anisotropic integro-differential operators, where the L\'{e}vy measure does not have to be absolutely continuous with respect to the Lebesgue measure.

This article studies regularity for operators in non-divergence form given by \eqref{gen:op}. 
H\"{o}lder regularity results have intensively been studied for linear and nonlinear nonlocal equations governed by operators in non-divergence form. 
\cite{SILVESTRE} provides a purely analytic proof of H\"older continuity for harmonic functions with respect to a class of integro differential equations given by \eqref{eq:integrodiffa}, where no symmetry on the kernel $a$ is assumed.
In \cite{CAFFSIL}, the authors study viscosity solutions to fully nonlinear integro-differential equations and prove a nonlocal version of the 
Aleksandrov-Bakelman-Pucci estimate, a Harnack inequality and a Hölder estimate. 
There are many more important results concerning H\"older estimates and Harnack inequalities for integro-differential equations
in non-divergence form including \cite{CafSilExt}, \cite{Imbert11}, \cite{Hector14}, \cite{CLU14}, \cite{Rang14} and \cite{ZHANG15}. 
H\"{o}lder regularity estimates have also been intensely studied for operators in divergence form. 
We would like to mention two works, where the corresponding jump intensities are similar to the ones we study in this article. 
In \cite{ChKa20} and \cite{chakawe19} the authors study nonlocal elliptic resp. parabolic equations for families of operators which can be of the form \eqref{eq:generator}. They prove a weak Harnack inequality and Hölder regularity estimates for weak solutions to the corresponding equations.

Let us give a short survey to known results related to systems of stochastic differential equations given by \eqref{system of equations}. 
We first discuss some results in the case $\alpha_1=\dots= \alpha_d$. In \cite{BCS} the authors prove unique weak solvability for \eqref{system of equations}. 
\cite{BCR} shows H\"{o}lder regularity estimates for bounded harmonic functions. 
Furthermore, in \cite{KRS18} the authors prove the strong Feller property for the corresponding semigroup for \eqref{system of equations}. 
Sharp lower bounds for the transition densities for the process $Z_t=(Z_t^1,\dots,Z_t^d)$ are studied in \cite{Xu13} and sharp upper bounds in \cite{KKK19}. \\
The existence of a unique solution to the martingale problem for \eqref{system of equations} in the case of different orders of differentiability, i.e. $\alpha_i\neq \alpha_j$ for $i\neq j$, is shown in \cite{ChMZ19} under the additional assumption that the matrix $A$ is diagonal.
\cite{Kulc} also studies the system \eqref{system of equations} in the case of diagonal matrices $A$. The authors prove 
sharp two-sided estimates of the corresponding transition density $p^{A}(t,x,y)$ and prove H\"older and gradient estimates for the function $x\mapsto p^{A}(t,x,y)$.
In \cite{FPR18} the authors study the existence of densities for solutions \eqref{system of equations} with H\"{o}lder continuous coefficients. They allow for a wide class of L\'{e}vy processes including the anisotropic processes $Z_t$ with different orders of differentiability.
In \cite{kulczsemigroup} the authors study systems of the form \eqref{system of equations} where $Z_t^1,\dots,Z_t^d$ are independent one-dimensional L\'{e}vy processes with characteristic exponents $\psi_1,\dots,\psi_d$. Under scaling conditions and regularity properties on the characteristic function they prove semigroup  properties for solutions.
 
\subsection*{Structure of the article}
This article is organized as follows. In \autoref{sectiondef} we provide definitions and auxiliary results. We constitute sufficient preparation and
study the behavior of the solution to the system. 
In \autoref{sectionsupport} we study the topological support of the solution to the martingale problem associated to the system of stochastic differential equations. 
The aim of this section is to prove a support theorem.
\autoref{sectionproof} contains the proof of \autoref{hoelder-reg}.

\section{Definitions and auxiliary results}\label{sectiondef}
In this section we provide important definitions and prove auxiliary results associated to the solution of the system \eqref{system of equations}.

Let $A^\tau(x)$ denote the transpose of the matrix $A(x)$ and $(a_j^\tau(x))^{-1}$ the j$^{\text{th}}$ row of $(A^{\tau}(x))^{-1}$. 
For a Borel set $D$, we denote the first entrance time of the process $X$ in $D$ by $T_D:=\inf\{t\geq 0 \colon X_t\in D\}$ and the first exit time of $X$ of $D$ by $\tau_D:=\inf\{t\geq 0 \colon X_t\notin D\}$.

Let us first recall the definition of harmonicity with respect to a Markov process.
\begin{definition}
A bounded function $h:\R^d\to\R$ is called harmonic with respect to $X$ in a domain $D\subset\R^d$ if for every bounded open set $U$ with $U\Subset D$  
\[ h\left(X_{t\wedge\tau_U}\right) \text{ is a } \mathbb{P}^x \text{-martingale for every } x\in U. \]
\end{definition}
 For $R=M_s(y)$ we use the notation $\widehat{R}=M^3_s(y)$. The next Proposition is a pure geometrical statement and not related to the system of stochastic differential equations. We skip the proof and refer the reader to \cite[Proposition V.7.2]{BAD}, which can be easily adjusted to our case.
\begin{proposition}\label{constructD}
 Let $r\in(0,1],q\in(0,1)$ and $x_0\in\R^d$. If $A\subset M_r(x_0)$ and $|A|<q$, then there exists a set $D\subset M_r(x_0)$ such that
 \begin{enumerate}
  \item $D$ is the union of rectangles $\widehat{R_i}$ such that the interiors of the $R_i$ are pairwise disjoint,
  \item $|A|\leq |D\cap M_r(x_0)|$ and
  \item for each $i$, $|A\cap R_i|>q|R_i|.$
 \end{enumerate}
\end{proposition}
Following the ideas of the proof of \cite[Proposition 2.3]{BLH}, we next prove a  L\'{e}vy system type formula. 
\begin{proposition}\label{Levy-system}
 Suppose $D$ and $E$ are two Borel sets with $\dist(D,E)>0$.
 Then
 \begin{equation*}
 \sum_{s\leq t} \mathds{1}_{\{X_{s-}\in D, X_s\in E\}} -\int_{0}^t \mathds{1}_D (X_s) \int _E \sum_{k=1}^d \left(\frac{|(a_k^\tau(X_s))^{-1}|^{1+\alpha_k}c_{\alpha_k}}{|h_k - X_s^{k}|^{1+\alpha_k}}dh_k\left(\prod_{j\neq k}\delta_{\{X_s^{j}\}}(dh_j) \right)\right) ds
 \end{equation*}
is a $\mathbb{P}^x$-martingale for each $x$.
 \end{proposition}
 \begin{proof}
 Let $f\in C_b^2(\R^d)$ with $f=0$ on $D$ and $f=1$ on $E$. Moreover set 
 \[M_{t}^f:= f(X_t) -f(X_0)-\int_{0}^{t} \mathcal{L}f(X_s) ds.\]
By Assumption (iii) for each $x\in\R^d$ the probability measure $\mathbb{P}^x$ is a solution to the martingale problem for $\mathcal{L}$. 
Since the stochastic integral with respect to a martingale is itself a martingale,
\[ \int_{0}^t \mathds{1}_D (X_{s-}) dM_s^f \]
is a $\mathbb{P}^x$-martingale.
Rewriting $ f(X_t)-f(X_0) = \sum_{s\leq t} (f(X_s)-f(X_{s-}))$ leads to 
\[ \sum_{s\leq t} \left( \mathds{1}_D (X_{s-}) (f(X_s)-f(X_{s-})) \right) - \int_{0}^t \mathds{1}_D (X_{s-}) \mathcal{L} f(X_s) ds\]
is a $\mathbb{P}^x$-martingale. Since $X_s\neq X_{s-}$ for only countably many values of $s$,
\begin{equation}\label{summartingale}
\sum_{s\leq t} \left( \mathds{1}_D(X_{s-}) (f(X_s)-f(X_{s-})) \right) - \int_{0}^{t} \mathds{1}_D (X_s) \mathcal{L}f(X_s) ds 
\end{equation}
is also a $\mathbb{P}^x$-martingale.
Let $w=(w_1,\dots,w_d)$ and $u=(u_1,\dots,u_d)$. By definition of $f$, for $x\in D$ we have $f(x)=0$ and $\nabla f(x)=0$. Hence
\begin{align*}
 \mathcal{L}f(x) & = \sum_{k=1}^{d} \int_{\R\setminus \{0\}} f(x+a_k(x)h) \frac{c_{\alpha_k}}{|h|^{1+\alpha_k}} dh \\
 & =\sum_{k=1}^d \int_{\R^d\setminus\{0\}} \left( f(x+A^{\tau}(x)w) \frac{c_{\alpha_k}}{|w|^{1+\alpha_k}}\left( \prod_{j\neq k} \delta_{\{0\}}(dw_j) \right) \right) dw_k \\ 
 & = \sum_{k=1}^d \int_{\R^d\setminus \{0\}} f(u) \frac{|(a_k^\tau(x))^{-1}|^{1+\alpha_k}c_{\alpha_k}}{|u-x|^{1+\alpha_k}}\left( \prod_{j\neq k} \delta_{\{x_j\}}(du_j) \right) du_k. \\
\end{align*}
Note, that $c_{\alpha_k}/|h|^{1+\alpha_k}$ is integrable over $h$ in the complement of any neighborhood of the origin for any $k\in\{1,\dots,d\}$. Since $D$ and $E$ have
a positive distance from each other, the sum in \eqref{summartingale} is finite. Hence
\begin{align*}
 \sum_{s\leq t}&\left( \mathds{1}_D (X_{s-}) (\mathds{1}_E (X_s)-\mathds{1}_E (X_s)) \right)\\
 &\qquad -  \int_{0}^t \mathds{1}_D (X_s) \int _E \sum_{k=1}^d \left(\frac{|(a_j^{\tau}(X_s))^{-1}|^{1+\alpha_j}c_{\alpha_k}}{|h_k - X_s^{k}|^{1+\alpha_k}}dh_k\left(\prod_{j\neq k}\delta_{\{X_s^{j}\}}(dh_j) \right)\right)ds 
\end{align*}
is a $\mathbb{P}^x$-martingale, which is equivalent to our assertion.
 \end{proof}
The next Proposition gives the behavior of the expected first exit time of the solution to \eqref{system of equations} out of the set $M_r(\cdot)$. 
This Proposition highlights the advantage of $M_r(\cdot)$ and shows that the scaling of the cube in the different 
directions with respect to the jump intensity compensates the different jump intensities in the different directions.
\begin{proposition}\label{exittime}
Let $x\in\R^d$ and $r\in(0,1].$ Then there exists a constant $c_1=c_1(\Lambda(M_r(x)),d)>0$ such that for all $z\in M_r(x)$
 \[ \mathbb{E}^{z}\left[ \tau_{M_r(x)}\right]\leq c_1 r^{\alpha_{\max}}. \]
 \end{proposition}
\begin{proof}[Proof]
First note
\begin{equation}\label{dto1}
\begin{aligned}
 \mathbb{E}^z \left[ \tau_{M_r(y)}\right]& =  \mathbb{E}^z \left[\min\limits_{1\leq i\leq d} \inf\{ t\geq 0 : X_t^i \notin (y_i - r^{(\alpha_{\max}/\alpha_i)},y_i - r^{(\alpha_{\max}/\alpha_i)} ) \} \right] \\
& \leq \frac{1}{d} \sum_{i=1}^d \mathbb{E}^z \left[ \inf\{ t\geq 0 : X_t^i \notin (y_i - r^{(\alpha_{\max}/\alpha_i)},y_i - r^{(\alpha_{\max}/\alpha_i)} ) \} \right] =: \frac{1}{d} \sum_{i=1}^d \mathbb{E}^z \left[\Upsilon_i\right].
\end{aligned} 
\end{equation}
Let $j\in\{1,\dots,d\}$ be fixed but arbitrary.
The aim is to show that there exists $c_2>0$ such that
\begin{equation}\label{Upsilon_j estimate}
 \mathbb{E}^z (\Upsilon_j) \leq c_2 r^{\alpha_{\max}}.
 \end{equation}
Since we reduced the problem to a one-dimensional one, we may suppose by scaling $r=1$.
Let 
\[\kappa:=\inf\left\{ |A(x)e_j|: x\in \overline{M_1(x)} \right\}.\]
By Assumption (i), we have $\kappa>0$. There exists a $c_3\in(0,1)$ with
 \[\mathbb{P}^z(\exists s\in[0,1]:\Delta Z_s^j \in \R\setminus[-3/\kappa,3/\kappa])\geq c_3.\]
The independence of the one-dimensional processes implies that with probability zero at least two of the $Z^i$'s make a jump at the same time. This leads to
\begin{equation}\label{leave-estimate}
 \mathbb{P}^z(\exists s\in[0,1]:\Delta Z_s^j>\frac{3}{\kappa} \text{ and } \Delta Z_s^i = 0 \text{ for } i\in\{1,\dots,d\}\setminus\{j\})\geq c_3.
\end{equation}
Our aim is to show that the probability of the process $X$ for leaving $M_1(x)$ in the j$^{\text{th}}$ coordinate after time $m$ is bounded in the following way
 \[\mathbb{P}^{z}(\Upsilon_j>m) \leq (1-k_j)^m \quad \text{for all } \ m\in\N.\]
Suppose there exists $s\in[0,1]$ such that $\Delta Z_s^j > \frac{3}{\kappa}$, $\Delta Z_s^i=0$ for $i\in\{1,\dots,d\}\setminus\{ j\}$, and
$X_{s-}\in M_1(x).$ Then
\[ |\Delta X_s^j|= |\Delta Z_s^j| \ |A(X_{s-})(e_j)| > 3.\]
Note, that we leave $M_1(x)$ by this jump.
By \eqref{leave-estimate}
\[ \mathbb{P}^{z}(\Upsilon_j\leq 1) \geq c_3 \Leftrightarrow \mathbb{P}^{z}(\Upsilon_j>1)\leq (1-c_3).\]
Let $\{\theta_t : t\geq 0\}$ denote the shift operators for $X$. Now assume $\mathbb{P}^{z}(\Upsilon_j>m) \leq (1-c_3)^m$. By the Markov property
\begin{align*}
 \mathbb{P}^{z}(\Upsilon_j>m+1) &  \leq \mathbb{P}^{z}(\Upsilon_j>m ; \Upsilon_j \circ \theta_m >1 ) \\
 & = \mathbb{E}^{z}[\mathbb{P}^{X_m}(\Upsilon_j > 1); \Upsilon_j > m] \\
 & \leq (1-c_3)\mathbb{P}^{z}(\Upsilon_j>m) \\
 & \leq (1-c_3)^{m+1}.
\end{align*}
Assertion \eqref{Upsilon_j estimate} follows by
\[\mathbb{E}^{x}[\Upsilon_j] = \int_{0}^{\infty} \mathbb{P}^z(\Upsilon_j>t)dt \leq \sum_{m=0}^\infty \mathbb{P}^z (\Upsilon_j>m)\leq \sum_{m=0}^\infty (1-c_3)^m = c_2, \]
where we used the fact that the sum on the right hand side is a geometric sum.
Thus the assertion follows by \eqref{dto1} and \eqref{Upsilon_j estimate}.
\end{proof}
We close this section by giving an estimate for leaving a rectangle with a comparatively big jump.
\begin{proposition}\label{enter-exit}
 Let $x\in\R^d$, $r\in(0,1]$ and $R\geq 2r$. There exists a constant $c_1=c_1(\Lambda(\R^d),d)>0$, such that for all $z\in M_r(x)$
 \[ \mathbb{P}^z (X_{\tau_{M_r(x)}}\notin M_R(x)) \leq c_1 \left(\frac{r}{R}\right)^{\alpha_{\max}}. \]
\end{proposition}
\begin{proof}[Proof]
Let
\[C_j:= \R\setminus [x_j-R^{\alpha_{\max}/\alpha_j},x_j+R^{\alpha_{\max}/\alpha_j}]\]
and for $1\leq j\leq d$ let $k_j = \sup_{x\in\R} |(a_j^\tau(x))^{-1}|c_{\alpha_j}$.
By \autoref{Levy-system} and optional stopping we get for $c_2=\sum_{j=1}^d ((2k_j2^\alpha_{\max})/c_{\alpha_j})\leq 8d \sup_{x\in\R} |(a_j^\tau(x))^{-1}|$
 \begin{align*}
  \mathbb{P}^z \left(X_{t \wedge \tau_{M_r(x)}}\notin M_R(x)\right)
  & = \mathbb{E}^z \left[\int_{0}^{t\wedge \tau_{M_r(x)}} \int_{M_R(x)^c} \sum_{j=1}^{d} \frac{|(a_j^\tau(X_s))^{-1}|c_{\alpha_j}}{|h_j-X_s^{j}|^{1+\alpha_j}} \left(\prod_{i\neq j}\delta_{\{X_s^{i}\}}(dh_i) \right) dh_j ds\right] \\
  & \leq \mathbb{E}^z \left[ \int_{0}^{t\wedge \tau_{M_r(x)}} \sum_{j=1}^d \int_{C_j} \frac{k_j}{|h_j-X_s^{j}|^{1+\alpha_j}} dh_j ds \right] \\
  & \leq \mathbb{E}^z \left[ \int_{0}^{t\wedge \tau_{M_r(x)}} \sum_{j=1}^d \int_{C_j} \frac{k_j}{|h_j-(x_j+r^{\alpha_{\max}/\alpha_j})|^{1+\alpha_j}} dh_j ds \right] \\
  & = \mathbb{E}^z [t\wedge\tau_{M_r(x))}] \sum_{j=1}^d \frac{2k_j}{\alpha_j (R^{\alpha_{\max}/\alpha_j}-r^{\alpha_{\max}/\alpha_j})^{\alpha_j}} \\
  &\leq \mathbb{E}^z [t\wedge\tau_{M_r(x))}] \sum_{j=1}^d \frac{2k_j}{\alpha_j ((R/2)^{\alpha_{\max}/\alpha_j})^{\alpha_j}} = \frac{c_2}{R^{\alpha_{\max}}}\mathbb{E}^z [t\wedge\tau_{M_r(x))}].
 \end{align*}
Using the monotone convergence on the right and dominated convergence on the left, we have for $t\to\infty$
\begin{align*}
\mathbb{P}^z(X_{t \wedge \tau_{M_r(x)}}\notin M_R(x)) \leq \frac{c_2}{R^{\alpha_{\max}}} \mathbb{E}^z (\tau_{M_r(x)}) \leq c_2 c_3\left(\frac{r}{R}\right)^{\alpha_{\max}},
\end{align*}
where $c_3$ is the constant showing up in the estimate $\mathbb{E}^z (\tau_{M_r(x)})\leq c_3 r^\alpha$ of \autoref{exittime}.
\end{proof}

\section{The support theorem}\label{sectionsupport}
In this section we prove the main ingredient for the proof of the H\"older regularity estimate for harmonic functions. The so-called support theorem 
states that sets of positive Lebesgue measure are hit with positive probability.

This theorem was first proved in \cite{KSA} for the diffusion case.
In the article \cite{BCR}, Bass and Chen prove the support theorem in the context of pure jump processes with singular and anisotropic kernels. 
They consider the system \eqref{system of equations} in the case $\alpha_i=\alpha$ for all $i\in\{1,\dots,d\}$
and use the technique by Krylov and Safonov to prove H\"older regularity with the help of the support theorem.\\
The idea we use to prove the support theorem is similar in spirit to the one in \cite{BCR}.

The following Lemma is a statement about the topological support of the law of the stopped process. It gives the existence of a bounded stopping time $T$
such that with positive probability the stopped process stays in a small ball around its starting point up to time $T$, makes a jump along the k$^{\text{th}}$ coordinate axis and 
stays afterwards in a small ball.
\begin{lemma}\label{Lemma1}
Let $r\in(0,1]$, $x_0\in\R^d, k\in\{1,\dots,d\}, v_k=A(x_0)e_k,$
$\gamma\in(0,r^{\alpha_{\max}/\alpha_{\min}}), t_0>0$ and $\xi\in[-r^{\alpha_{\max}/\alpha_{\min}},r^{\alpha_{\max}/\alpha_{\min}}]$.
There exists a constant $c_1>0=c_1(\gamma,t_0,\xi,r,\Lambda(M^{2}_r(x_0))),\m(M^{2}_r(x_0)))>0$ and a stopping time $T\leq t_0$,
 such that
 \begin{equation}
  \begin{aligned}
   \mathbb{P}^{x_0} \left(\sup\limits_{s<T}|X_s-x_0|<\gamma \text{ and } \sup\limits_{T\leq s \leq t_0} |X_s - (x_0+\xi v_k)|<\gamma\right)\geq c_1.
  \end{aligned}
 \end{equation}
 \end{lemma}
 \begin{proof}
 Let \[\|A\|_\infty := 1 \vee \left(\sum_{i,j=1}^{d} \sup\limits_{x\in M^{2}_r(x_0)} |A_{ij}(x)|\right).\]
 We assume $\xi\in[0,r^{\alpha_{\max}/\alpha_{\min}}]$. The case $\xi\in[-r^{\alpha_{\max}/\alpha_{\min}},0]$ can be proven similar.
 Let us first suppose $\xi\geq \gamma /(3\|A\|_\infty)$ and let $\beta\in (0,\xi)$, which will be chosen later.
 We decompose the process $Z_t^i$ in the following way:
 \[ \widetilde{Z}_t^i = \sum_{s\leq t} \Delta Z_s^i \mathds{1}_{\{|\Delta Z_s^i|>\beta\}}, \quad \overline{Z}_t^i = Z_t^i - \widetilde{Z}_t^i. \]
 Let $(\overline{X}_t)_{t\geq0}$ be the solution to  
 \[ d\overline{X}^i_{t} = \sum_{j=1}^d A_{ij}(\overline{X}_{t-})d\overline{Z}_{t}^j, X_0^i = x_0^i.\]
 The continuity of $A$ allows us to find a $\ \delta< \gamma / (6\|A\|_\infty)$, such that
 \begin{equation}\label{supsup-estimate}
  \begin{aligned}
   \sup\limits_{i,j}\sup\limits_{|x-x_0|<\delta} |A_{ij}(x)-A_{ij}(x_0)|<\frac{\gamma}{12d}.
  \end{aligned}
 \end{equation}
Consider
\begin{align*}
 C &= \left\{ \sup\limits_{s\leq t_0}|\overline {X}_s - \overline{X}_0|\leq \delta \right\}, \\
 D &= \Big\{ \widetilde{Z}^k \text{ has precisely one jump before time } t_0 \text{ with jump size in } [\xi,\xi+\delta],\\
 & \qquad  \Delta\widetilde{Z}_s^j = 0 \text{ for all } s\leq t_0 \text{ and all } j\neq k \Big\}, \\
 E &= \left\{ \widetilde{Z}_s^i=0 \text{ for all } s\leq t_0 \text{ and } i=1,\dots,d \right\}.
\end{align*}
Since $A$ is bounded, we can find $c_2>0$, such that
\[ \left[ \overline{X}^i, \overline{X}^i \right]_t\leq c_2 \sum_{j=1}^d \left[\overline{Z}^j, \overline{Z}^j\right]_t.\]
Note, that $\beta\in(0,\xi)\subset(0,r^{\alpha_{\max}/\alpha_{\min}})\subset (0,1)$. Therefore, we get
\begin{align*}
\mathbb{E}^{x_0}\left[\overline{X}^i,\overline{X}^i\right]_t &\leq c_2 \sum_{j=1}^d \mathbb{E}^{x_0}\left[\overline{Z}^j,\overline{Z}^j\right]_t = c_2 \sum_{j=1}^d \int_{0}^t\left( \int_{-\beta}^{\beta} \frac{c_{\alpha_j}h^2}{|h|^{1+\alpha_j}} dh\right)dt  \leq c_3 t d \beta^{2-\alpha_{\max}}.
 \end{align*}
By Tschebyscheff's inequality and Doob's inequality, we get
\[ \mathbb{P}^{x_0}\left[ \sup\limits_{s\leq t_0} |\overline{X}_s^i-\overline{X}_0^i|>\delta \right] \leq \frac{1}{\delta^2}\mathbb{E}^{x_0}\left[ \sup\limits_{s\leq t_0} \left( \overline{X}_s^i - \overline{X}_0^i \right)^2 \right]
\leq \frac{1}{\delta^2} 4 \mathbb{E}^{x_0}\left[\left( \overline{X}_{t_0}^i - \overline{X}_0^i \right)^2\right] \leq \frac{c_4 t_0 d \beta^{2-\alpha_{\max}}}{\delta^2}. \]
Choose $\beta\in (0,\xi)$ such that
\begin{equation}\label{beta-estimate}
 \begin{aligned}
  c_5 t_0 \beta^{2-\alpha_{\max}} \leq \frac{\delta^2}{2d}
 \end{aligned}
\end{equation}
holds. Then by \eqref{beta-estimate}, we get
\begin{equation}\label{estimate-C}
 \mathbb{P}^{x_0}(C) = 1- \mathbb{P}^{x_0} \left( \sup\limits_{s\leq t_0} |\overline{X}_s^i-\overline{X}_0^i|>\delta \right) \geq \frac{1}{2}.
\end{equation}
For $\widetilde{Z}^k$ to have a single jump before time $t_0$,
and for that jump's size to be in the interval
$[\xi,\xi+\delta]$, then up to time $t_0$ $\widetilde{Z}^k_t$ must have {\ }\vspace*{-1.3em}
\begin{enumerate}
 \item[(i)] no negative jumps, \label{1}
 \item[(ii)] no jumps whose size lies in $[\beta,\xi),$ \label{2}
 \item[(iii)] no jumps whose size lies in $(\xi+\delta,\infty),$ \label{3}
 \item[(iv)] precisely one jump whose size lies in the interval $[\xi,\xi+\delta]$. \label{4}
\end{enumerate}
We can use the fact, that $\widetilde{Z}^k$ is a compound Poisson process and use the knowledge about Poisson random measures.
The events descriped in (i)-(iv) are the probabilities that Poisson
random variables $ P_1,P_2.P_3$ and $P_4$ of parameters $\lambda_1=c_6 t_0 \beta^{-\alpha_k}$, $\lambda_2=c_6 t_0 (\beta^{-\alpha_k}-\xi^{-\alpha_k})$,
$\lambda_3=c_6 t_0 (\xi+\delta)^{-\alpha_k}$,
and $\lambda_4=c_6 t_0 (\xi^{-\alpha_k} - (\xi+\delta)^{-\alpha_k})$, respectively, take the values $0,0,0,$ and $1$, respectively.  \newline
So there exists a constant $c_7=c_7(\alpha_k,t_0,\delta, \xi, \beta)>0$ such that
\begin{align*}
 \mathbb{P}^{x_0}\big( &\widetilde{Z}^k \text{ has a single jump before time $t_0$, and its size is in } [\xi,\xi+\delta]\big) \geq c_7.
\end{align*}
For all $j\neq k$, the probability that $\widetilde{Z}^{j}$ does not have a jump before time $t_0$, is the probability that a Poisson random variable
with parameter $2c_6 t_0\beta^{-\alpha_j}$ is equal to $0$.
Using the indepence of $\widetilde{Z}^j$ for $j=1,\dots,d$, we can find a $c_8>0$ such that
\begin{align*}
 \mathbb{P}^{x_0}(\Delta\widetilde{Z}^j_s = 0 \text{ for all } s\leq t_0 \text{ and all } j\neq k) \geq c_8.
\end{align*}
Thus we obtain
\[ \mathbb{P}^{x_0}(D)\geq c_9 \]
for a $c_9=c_9(\alpha_1,\dots,\alpha_d,t_0,\delta, \xi, \beta)>0$.
Furthermore the $\overline{Z}^i$'s are independent of the $\widetilde{Z}^j$'s for all $i,j\in\{1,\dots,d\}$,
so $C$ and $D$ are independent and we obtain
\begin{align*}
 \mathbb{P}^{x_0}(C\cap D)\geq c_9/2.
\end{align*}
Similary we obtain
\begin{align}\label{C cap E}
\mathbb{P}^{x_0}(E)\geq c_{10} \quad \text{and} \quad \mathbb{P}^{x_0}(C\cap E)\geq c_{11}.
\end{align}\newline
Let $T$ be the time, when $\widetilde{Z}^k$ jumps the first time, i.e. $Z^k$ makes a jump greater then $\beta$.
Then $Z_{s-}=\overline{Z}_{s-}$ for all $s\leq T$ and hence $X_{s-}=\overline{X}_{s-}$ for all  $s\leq T.$
So up to time $T$, $X_s$ does not move away more than $\delta$ away from its starting point. Note
$\Delta X_T = A(X_{T-})\Delta Z_T.$
By \eqref{supsup-estimate}, we obtain on $C\cap D$
\begin{align*}
 |X_T-(x_0+\xi v_k)|
 & \leq |X_{T-}-x_0| + |\Delta X_T - \xi A(x_0)e_k))| \\
 & = |X_{T-}-x_0| + | A(X_{T-})\Delta Z_T - \xi A(x_0)e_k) | \\
 & \leq |X_{T-}-x_0| + \xi |(A(X_{T-})-A(x_0))e_k| + |A(X_{T-})(\Delta Z_T -\xi e_k)| \\
 & \leq \delta + \frac{\xi d \gamma}{12d} + \delta \|A\|_{\infty} \leq \frac{\gamma}{6} \left( \frac{1}{\|A\|_\infty} + \frac{\gamma}{2} + 1\right) \leq \frac{\gamma}{2}.
\end{align*}
Appling the strong Markov property at time $T$, we get by \eqref{C cap E}
\begin{align*}
\mathbb{P}^{x_0}\left( \sup\limits_{T\leq s\leq T+t_0} |X_s-X_T|<\delta \right) \geq\mathbb{P}^{X_T} (C\cap E) \geq c_{11}.
\end{align*}
Note, that $ |X_T-(x_0+\xi v_k)|<\gamma/2$ and $|X_s-X_T|<\delta$ for all $T\leq s \leq t_0$ imply $|X_s-(x_0+\xi v_k)|<\gamma.$\newline
All in all we get by the strong Markov property
\[    \mathbb{P}^{x_0} \left(\sup\limits_{s<T}|X_s-x_0|<\gamma \text{ and } \sup\limits_{T\leq s \leq t_0} |X_s - (x_0+\xi v_k)|<\gamma\right)\geq \frac{c_{9}c_{11}}{2}, \]
which proves the assertion. \newline
Now suppose $\xi <\gamma/(3\|A\|_\infty$). Then $|x_0-(x_0+\xi v_k)|<\gamma/3$. We can choose $T\equiv0$ and by \eqref{C cap E} we get:
\[ \mathbb{P}^{x_0}\left( \sup\limits_{s\geq t_0} |X_s-x_0|<\delta \right)\geq c_{11}, \]
which finishes the proof.
\end{proof}
We need two simple geometrical facts from the field of linear algebra, whose proofs can be found in \cite{BCR} (Lemma 2.4 and Lemma 2.5).
\begin{lemma}\label{linearalgebra1}
 Suppose $u,v$ are two vectors in $\R^d$, $\eta\in(0,1)$, and $p$ is the projection of $v$ onto $u$.
 If $|p|\geq \eta |v|$, then
 \[ |v-p|\leq \sqrt{1-\eta^2}|v|. \]
\end{lemma}

\begin{lemma}\label{linearalgebra2}
 Let $v$ be a vector in $\R^d$, $u_k = Ae_k$, and $p_k$ the projection
 of $v$ onto $u_k$ for $k=1,\dots,d.$ Then there exists $\rho=\rho(\Lambda(\R^d))\in (0,1)$, such that for some $k$,
 \[ |v-p_k|\leq \rho |v|.\]
\end{lemma}
For a given time $t_1>0$ the following lemma shows that solutions stay with positive probability
in an $\varepsilon$-tube around a given line segment on $[0,t_1]$. The case of $\alpha_1=\cdots=\alpha_d$ was considered in \cite{BCR}. We follow their technique.
\begin{lemma}\label{support2}
Let $r\in(0,1]$, $x_0\in\R^d$, $t_1>0, \varepsilon\in(0,r^{\alpha_{\max}/\alpha_{\min}}), \xi\in(0,\varepsilon/4)$ and $\gamma>0.$
Moreover let $\psi:[0,t_1]\to\R^d$ be a line segment of length $\xi$ starting at $x_0$.
Then there exists $c_1=c_1(\Lambda(M_r^2(x_0))), \m(M_r^2(x_0))),t_1, \varepsilon, \gamma)>0$, such that
 \[ \mathbb{P}^{x_0}\left( \sup\limits_{s\leq t_1} |X_s - \psi(s)|<\varepsilon \text{ and } |X_{t_1}-\psi(t_1)|<\gamma \right)\geq c_1. \]
\end{lemma}
\begin{proof}
 Note that $\varepsilon$ is chosen such that $B_{\varepsilon}(x_0)\subset M_r(x_0)$.
 Let $\rho\in(0,1)$ be such that the conclusion of \autoref{linearalgebra2} holds for all matrices $A=A(x)$ with $x\in M_r^{2}(x).$
 Take $\gamma \in (0,\xi \wedge \rho)$ such that $\widetilde{\rho}:=\gamma + \rho<1$ and $n\geq 2$ sufficiently large, such that
  $(\widetilde{\rho})^n < \gamma.$
 Let $v_0:=\psi(t_1)-\psi(0)=\psi(t_1)-x_0,$ which has length $\xi.$
 By \autoref{linearalgebra1}, there exists a $k_0\in\{1,\dots,d\}$
 such that if $p_0$ is the projection of $v_0$ onto $A(x_0)e_{k_0}$, then $|v_0-p_0|\leq \rho|v_0|.$
 Note, that $|p_0|\leq |v_0|=\xi.$
 By \autoref{Lemma1} there exists $c_2>0$ and a stopping time $T_0\leq t_1/n$ such that for
\begin{align*}
D_1:= \left\{\sup\limits_{s<T_0}|X_s-x_0|<\gamma^{n+1} \text{ and } \sup\limits_{T_0\leq s \leq t_1/n}|X_s - (x_0+p_0)|<\gamma^{n+1} \right\}.
 \end{align*}
the estimate \[\mathbb{P}^{x_0}(D_1)\geq c_2\]
holds.
Since $\gamma<1$ and $\gamma^n\leq \gamma $ for all $n\in\N$, we have for $T_0 \leq s \leq t_1/n$
\begin{equation}\label{psi}
 \begin{aligned}
  |\psi(t_1)-X_s| & \leq|\psi(t_1)-(x_0+p_0)| + |(x_0+p_0)-X_s| \\
  & \leq |v_0-p_0| + \gamma^{n+1}  = \rho \xi + \gamma^{n+1} \leq \widetilde{\rho}\xi
 \end{aligned}
\end{equation}
on $D_1$. Taking $s=t_1/n$, we have
\[ |\psi(t_1)-X_{t_1/n}|\leq \widetilde{\rho}\xi. \]
Since $\widetilde{\rho}<1$ and $|\psi(t_1)-x_0|=|v_0|=\xi,$ then \eqref{psi} shows that on $D_1$
\[ X_s\in B(x_0,2\xi)\subset B(x_0, \varepsilon/2) \quad \text{if } T_0\leq s\leq t_1/n. \]
If $0\leq s<T_0,$ then $|X_s-x_0|<\gamma^{n+1}<\xi,$ and so we have on $D_1$
\[\{ X_s, s\in [0,t_1/n] \}\subset B(x_0,2\xi)\subset B(x_0,\varepsilon/2).\]
Now let $v_1:=\psi(t_1)-X_{t_1/n}.$
When $X_{t_1/n}\in B(x_0, \varepsilon/2),$ then by \autoref{linearalgebra2}, there exists
$k_1\in\{1,\dots,d\}$ such that if $p_1$ is the projection of $v_1$ onto $A(X_{t_1/n})e_{k_1},$ then
$ |v_1-p_1|\leq \rho|v_1|.$
Let $T_1\in [t_1/n, 2t_1/n]$ be a stopping time, determined by \autoref{Lemma1}, and
\begin{align*}
D_2:= \left\{ \sup\limits_{t_1/n\leq s <T_1}|X_s-X_{t_1/n}|<\gamma^{n+1}
\text{ and } \sup\limits_{T_1\leq s \leq 2t_1/n}|X_s - (X_{t_1/n}+p_1)|<\gamma^{n+1} \right\}.
\end{align*}
By the Markov property at the time $t_1/n$ and \autoref{Lemma1}, there exists the same $c_2>0$ such that
\[ \mathbb{P}^{x_0}(D_2|\mathcal{F}_{t_1/n})\geq c_2\]
on the event $\{ X_{t_1/n}\in B(x_0, \varepsilon/2) \}$ and hence especially on $D_1$.
So
\[ \mathbb{P}^{x_0}(D_1\cap D_2) \geq c_2 \mathbb{P}^{x_0}(D_1)\geq c_2^2.\]
Let $s\in[T_1,2t_1/n]$. Then on $D_1\cap D_2$
\begin{align*}
 |\psi(t_1)-X_s| & \leq |\psi(t_1)-(X_{t_1/n}+p_1)|+|(X_{t_1/n}+p_1)-X_s| \\
 & \leq \rho |v_1| + \gamma^{n+1} \leq \rho \widetilde{\rho}\xi+\gamma^{n+1}\leq (\widetilde{\rho})^2 \xi + \gamma^{n+1}\\
 & \leq \widetilde{\rho}^2\xi.
\end{align*}
In particular, by choosing $s=2t_1/n$, we get on $D_1\cap D_2$
\[ |\psi(t_1)-X_{2t_1/n}| \leq (\widetilde{\rho})^2\xi.\]
On $D_1\cap D_2$, we have for $s\in[T_1,2t_1/n]$: $|\psi(t_1)-X_s|<\xi$ and $|\psi(t_1)-x_0|=\xi$, which implies
\[ X_s\in B(x_0,2\xi)\subset B(x_0,\varepsilon/2)\quad \text{ on } D_1\cap D_2.\]
In particular,
\[ |X_{2t_1/n}-x_0|<2\xi \quad \text{ on } D_1\cap D_2.\]
If $s\in[t_1/n,T_1]$, then $|X_s-X_{t_1/n}|<\xi$ and $ |X_{t_1/n}-x_0|<2\xi$ on $D_1\cap D_2$, which yields to
\[X_s\in B(x_0,3\xi)\subset B(x_0,3\varepsilon/4)\quad \text{ on } D_1\cap D_2.\]
Let $v_2:=\psi(t_1)-X_{2t_1/n},$ and proceed as above to get events $D_3,\dots,D_k$ for $k\leq n$. At the k$^{\text{th}}$ stage
\[ \mathbb{P}^{x_0}(D_k | F_{(k-1)t_1/n})\geq c_2 \quad \text{ and so } \quad  \mathbb{P}^{x_0}\left(\bigcap_{j=1}^k D_j\right)\geq c_2^k.\]
For $kt_1/n \leq T_k \leq s \leq (k+1)t_1/n$
\[ |\psi(t_1)-X_s|\leq (\widetilde{\rho})^{k+1}\xi<\xi; \]
on the event $\bigcap_{j=1}^k D_j.$ Since $|\psi(t_1)-x_0|=\xi.$
\[ X_s\in B(x_0,2\xi)\subset B(x_0,\varepsilon/2)\quad \text{on } \bigcap_{j=1}^k D_j.\]
If $kt_1/n \leq s < T_k$, we obtain at the k$^{\text{th}}$ stage
\[ |X_{kt_1/n}-x_0|<\varepsilon/2 \quad \text{on } \bigcap_{j=1}^k D_j.\]
Thus
\[ |X_s-x_0|\leq |X_s-X_{kt_1/n}|+|X_{kt_1/n}-\psi(t_1)|+|\psi(t_1)-x_0|\leq \gamma^{n+1}+\xi+\xi< 3\xi, \]
and therefore $X_s\in B(x_0,3\xi)\subset B(x_0,3\varepsilon/4).$
We continue this procedure $n$ times to get events $D_1,\dots,D_n$.
On $\bigcap_{k=1}^n D_k$, we have
\begin{enumerate}
 \item $X_s\in B(x_0,3\xi)$ for $ s\leq t_1$,
 \item $|X_{t_1}-\psi(t_1)|<(\widetilde{\rho})^n\xi<(\widetilde{\rho})^n<\gamma, $ and
 \item $\displaystyle\mathbb{P}^{x_0}\left(\bigcap_{j=1}^n D_j\right)\geq c_2^n.$
\end{enumerate}
For $s\in[0,t_1]$,
\[ |X_s-\psi(s)|\leq |X_s-x_0|+|x_0-\psi(s)|<3\xi+\xi< \frac{3}{4}\varepsilon+ \frac{1}{4}\varepsilon=\varepsilon \quad \text{ on } \bigcap_{j=1}^n D_j. \]
Hence
\[ \mathbb{P}^{x_0}\left( \sup\limits_{s\leq t_1} |X_s - \psi(s)|<\varepsilon \text{ and } |X_{t_1}-\psi(t_1)|<\gamma \right) \geq \mathbb{P}^{x_0}\left(\bigcap_{j=1}^n D_j\right)\geq c_2^n=: c_1.\]
\end{proof}
We can now prove an important theorem, which will be the main ingredient in the proof of the H\"older regularity. It states that the solution to \eqref{system of equations} stays with positive probability in 
a $\varepsilon$-tube around a given continuous function.
\begin{theorem}\label{support-theorem}
Let $r\in(0,1], x_0\in\R^d, \varepsilon\in(0,r^{\alpha_{\max}/\alpha_{\min}}), t_0>0$ and $x_0\in\R^d$.
Let $\varphi:[0,t_0]\to\R^d$ be continuous with $\varphi(0)=x_0$ and the image of $\varphi$ contained in $M_r(x_0).$
Then there exists $c_1=c_1(\Lambda(M^2_r(x_0)),\m(M^3_r(x_0)),\varphi,\varepsilon, t_0)>0$ such that
 \[ \mathbb{P}^{x_0}\left( \sup\limits_{s\leq t_0} |X_s - \varphi(s)|<\varepsilon \right) > c_1. \]
\end{theorem}
\begin{proof}
Let $\varepsilon>0.$ We define
\[ U:=\{x\in\R^d : \exists s\in[0,t_0] \text{ such that } |x-\varphi(s)|<\varepsilon/2  \}\]
and approximate $\varphi$ within $U$ by a polygonal path.
Hence we can assume that $\varphi$ is polygonal by changing $\varepsilon$ to $\varepsilon /2$ in the assertion.
We subdivide $[0,t_0]$ into $n$ subintervals of the same length for $n\geq 2$ such that
\[r:= L\left( \varphi\left(\left( \frac{kt_0}{n},\frac{(k+1)t_0}{n} \right)\right) \right)<\frac{\varepsilon}{4},\]
where $L$ denotes the length of the line segment. Let
\[D_k := \left\{ \sup\limits_{(k-1)t_0/n \leq s \leq kt_0 / n} |X_s-\varphi(s)|<\frac{\varepsilon}{2} \text{ and } |X_{kt_0/n}-\varphi(kt_0/n)|<\frac{\varepsilon}{4\sqrt{d}} \right\}.\]
Using \autoref{support2}, there exists a constant $c_2>0$ such that
\[ \mathbb{P}^{x_0} (D_1)\geq c_2.\]
By the strong Markov property at time $t_0/n$ we get
\[\mathbb{P}^{x_0} (D_2|\mathcal{F}_{t_0/n})\geq c_2.\]
Using the Iteration as in the proof of \autoref{support2}, we get for all $k\in\{1,\dots,d\}$
\[ \mathbb{P}^{x_0} (D_k|\mathcal{F}_{(k-1)t_0/n})\geq c_2 \quad \text{and} \quad \mathbb{P}^{x_0} \left( \bigcap_{k=1}^n D_k \right)\geq c_2^n. \]
Hence the assertion follows by
\begin{align*}
 \mathbb{P}^{x_0}\left( \sup\limits_{s\leq t_0} |X_s - \varphi(s)|<\frac{\varepsilon}{2} \right) \geq \mathbb{P}^{x_0}\left(\bigcap_{k=1}^n D_k\right)\geq c_2^n=c_1.
\end{align*}
 \end{proof}
We state two corollaries, which follow immediately from \autoref{support-theorem}. 
 \begin{corollary}\label{support-corollary1}
 Let $r\in(0,1]$, $\varepsilon\in (0,r^{\alpha_{\max}/\alpha_{\min}}/4)$, $k=1-(\varepsilon/r^{\alpha_{\max}/\alpha_{\min}})$, $\delta\in(\varepsilon,r^{\alpha_{\max}/\alpha_{\min}}/2)$ and $x_0\in\R^d$. 
 Moreover let $Q:=M_r(x_0), Q':=M_r^{k}(x_0)$ and $y\in\R^d$ such that $R:=M_r^{\delta}(y)\subset Q'$.
 There exists $c_1=c_1(\Lambda(Q),\m(Q),\varepsilon, \delta)>0$  such that
 \[ \mathbb{P}^{x}(T_R<\tau_Q)\geq c_1, \quad x\in Q'. \]
\end{corollary}
\begin{proof}
Note, that
\[ \dist(\partial Q, \partial Q') = \left|r^{\alpha_{\max}/\alpha_i} - kr^{\alpha_{\max}/\alpha_i}\right| = |\varepsilon \frac{r^{\alpha_{\max}/\alpha_i}}{r^{\alpha_{\max}/\alpha_{\min}}}|\geq \varepsilon. \]
  Let $x\in Q'$ be arbitrary and $\varphi:[0,t_0]\to \R^d$
  be a polygonal path such that $\varphi(0)=x$ and $\varphi(t_0)=y$ and the image of $\varphi$ is contained in $Q'$.
 Then the assertion follows by \autoref{support-theorem} and
\[ \mathbb{P}^x \left( \sup\limits_{s\leq t_o} |X_s-\varphi(s)|<\varepsilon \right)\leq \mathbb{P}^x (T_R<\tau_Q).\]
 \end{proof}
\begin{corollary}\label{support-corollary2}
Let $r\in(0,1]$, $x_0\in\R^d$ and $\varepsilon\in(0,r^{\alpha_{\max}/\alpha_{\min}}/4).$ 
For $x\in M_r(x_0)$, we define $R:=M_s(x)$ such that $R\subset M_r(x_0)=:M$ and $\dist(\partial R, \partial M)>\varepsilon$.
Then there exists $\xi=\xi(\varepsilon, \Lambda(M), r, \m(M)\in(0,1)$ such that 
\[ \mathbb{P}^y(T_{R}<\tau_{M})\geq \xi=\xi(\epsilon).\]
for all $y\in M$ with $\dist(y,\partial M)>\varepsilon$.
 \end{corollary}
\begin{proof}
Follows immediately by \autoref{support-corollary1}.
 \end{proof}
 We now prove the main ingredient for the proof of the H\"older regularity. It states, that sets of positive Lebesgue measure are hit with positive probability.
 \begin{theorem}\label{supportphi}
 Let $r\in(0,1]$, $x_0\in\R^d$ and $M:=M_r(x_0).$
 There exists a nondecreasing function $\varphi:(0,1)\to(0,1)$ such that
 \[ \mathbb{P}^x (T_A<\tau_{M})\geq \varphi(|A|) \]
 for all $x\in M_r^{1/2}(x_0)$ and all $A\subset M$ with $|A|>0$.
\end{theorem}
\begin{proof}
 We will follow the proof of Theorem V.7.4 in \cite{BAD}. Set
 \begin{align*}
 \varphi(\varepsilon)=\inf \left\{  \mathbb{P}^y (T_A<\tau_{M_R(z_0)}) : z_0\in\R^d, R>0, y\in M_R^{1/2}(z_0), |A|\geq \varepsilon |M_R(z_0)|, A\subset M_R(z_0). \right\}
 \end{align*}
 and
 \[q_0:=\inf\{ \varepsilon : \varphi(\varepsilon)>0 \}.\]
For $\epsilon$ sufficiently large, we know by \autoref{support-corollary1} $\varphi(\epsilon)>0$.
 We suppose $q_0>0$, and we will obtain our contradiction. \newline
 Since $q_0<1$, we can choose $1>q>q_0$ such that $ (q+q^2)/2<q_0$. Moreover let
 \[\eta:= (q-q^2)/2, \quad \beta:= \left(\frac{2^{1-d} q r^{-\sum_{i=1}^d (\alpha_{\max}/\alpha_i)}}{q+1}\right)^{1/d} \quad \text{and} \quad \rho=\xi((1-\beta)r^{\alpha_{\max}/\alpha_{\min}}/6),\] 
 where $\xi$ is defined as in \autoref{support-corollary2}. \newline
 Let $z\in\R^d, R\in(0,1]$, $x\in M_R^{1/2}(z)$ and $A\subset M_R(z)$ such that
 \begin{equation} \label{qeta}
 q-\eta < \frac{|A|}{|M_R(z)|}<q \ \text{ and } \ \mathbb{P}^x(T_A<\tau_{M_R(z)})<\rho\varphi(q)^2. 
 \end{equation}
Without loss of generality, set $R=r$ and $z=x_0$. Hence
\[  \mathbb{P}^x(T_A<\tau_{M_r(x_0)})<\rho\varphi(q)^2. \]
By \autoref{constructD} there exists a set $D\subset M_r(x_0)$ such that
 \[ |A|\leq q|D\cap M_r(x_0)|. \]
Since $|A|>(q-\eta)|M_r(x_0)|$, 
\[ |D\cap M_r(x_0)| \geq \frac{|A|}{q}>\frac{(q-\eta)|M_r(x_0)|}{q} = \frac{(q+1)|M_r(x_0)|}{2}.\]
Define $E=D\cap M_r^\beta(x_0)$. Since
\[ \frac{(q+1)|M_{r}^\beta(x_0)|}{2}=\frac{(q+1)2^d \beta^d r^{\sum_{i=1}^d \frac{\alpha_{\max}}{\alpha_i}}}{2} = q, \]
we get $|E|>q.$ By the definition of $\varphi$, we have $\mathbb{P}^x(T_{E}<\tau_{M_r(x_0)})\geq \varphi(q).$\\
We will first show
\begin{equation}
 \mathbb{P}^y (T_A<\tau_{M_r(x_0)})\geq \rho\varphi(q) \quad \text{for all } y\in E.
\end{equation}
Let $y\in\partial E$, then $y\in\widehat{R}_i$ for some $R_i\in\mathcal{R}$ and $\dist(y,\partial M_r(x_0))\geq (1-\beta)r^{\alpha_{\max}\alpha_{\min}}$.
Define $R_i^*$ as the cube with the same center as $R_i$ but sidelength half as long. 
By \autoref{support-corollary2}
\[ \mathbb{P}^y (T_{R_i^*}<\tau_{M_r(x_0)})\geq \rho. \]
By \autoref{constructD} (3) for all $R_i\in\mathcal{R}$ 
\[|A\cap R_i|\geq q|R_i|\]
and therefore
\[\mathbb{P}^{x_0}(T_{A\cap R_i}<\tau_{M_r(x_0)})\geq \varphi(q) \quad \text{ for } x_0\in R_i^*.\]
Using the strong Markov property, we have for all $y\in E$
\begin{align*}
 \mathbb{P}^{y}(T_A<\tau_{M_r(x_0)}) & \geq \mathbb{E}^{y} \left[\mathbb{P}^{X_{T_{R_i^*}}}(T_A<\tau_{R_i}); T_{R_i^*}<\tau_{M_r(x_0)}\right] \\
 & \geq \rho \varphi(q).
\end{align*}
Now we get our contradiction by
\begin{align*}
 \mathbb{P}^x(T_A<\tau_{M_r(x_0)}) & \geq \mathbb{P}^x (T_{E}<T_A<\tau_{M_r(x_0)}) \\
 & \geq \mathbb{E}^x \left[\mathbb{P}^{X_{T_{E}}}(T_A < \tau_{M_r(x_0)});T_{E}<\tau_{M_r(x_0)}\right] \\
 & \geq \rho \varphi(q) \mathbb{P}^x \left[T_{E}<\tau_{M_r(x_0)}\right] \geq \rho \varphi(q)^2.
\end{align*}
 \end{proof}
\section{Proof of \autoref{hoelder-reg}}\label{sectionproof}
In this section we prove our main result.
\begin{proof}
 Let $S:=M_s(y)\subset M_r(x_0)$ and $A\subset S$ such that $3|A|\geq |S|.$ Those sets will be specified later.
 Set $k=1-(\varepsilon / r^{\alpha_{\max}/\alpha_{\min}})$ and $S':=M_s^k(y)$, 
 where $\varepsilon$ is chosen such that $6|S\setminus S'|=|S|.$ Then
\[ 6|A\cap S'| \geq |S|. \]
Let $\mathcal{R}$ be a collection of $N$ equal sized rectangles as in \autoref{rectangles} with disjoint interiors
and $\mathcal{R}\subset S$. Moreover let $\mathcal{R}$ to be a covering of $S'$.
For at least one rectangle $Q\in\mathcal{R}$
\[ 6|A\cap S' \cap Q|\geq  |Q|. \]
Let $Q'$ be the rectangle with the same center as $Q$ but each sidelength half as long. 
By \autoref{support-corollary1} there exists a $c_2>0$ such that
\begin{equation}
 \mathbb{P}^x(T_{Q'}<\tau_S) \geq c_2, \quad x\in M_s^{1/2}(y).
\end{equation}
Using \autoref{supportphi} and the strong Markov property there exists a constant $c_3>0$ with
\begin{equation}\label{enter-exit1}
 \mathbb{P}^x(T_A<\tau_S)\geq c_3, \quad x\in M_s^{1/2}(y).
\end{equation}
Let $R\geq 2r$. By \autoref{enter-exit} there exists a $c_4>0$ such that
 \[ \mathbb{P}^z \left( X_{\tau_{M_r(x_0)}}\notin M_R(x_0) \right)\leq c_4 \left( \frac{r}{R} \right)^\alpha_{\max} \quad \text{for all } z\in M_r(x_0). \]
Let
\[\gamma:= (1-c_3), \quad \text{ } \quad \rho:=\left( \frac{c_3\gamma^2}{4c_4} \right)^{1/\alpha_{\max}}\wedge \left( \frac{\gamma}{2} \right)^{1/\alpha_{\max}} \quad \text{and} \quad \beta:= \frac{\log(\gamma)}{\log(\rho)}. \]
By linearity it suffices to suppose $0\leq h \leq M$ on $\R^d.$ 
We first consider the case $r=1$. \newline
Let $M_i = M_{\rho^i}(x_0)$ and $\tau_i = \tau_{M_i}.$
We will show that for all $k\in\N_{0}$
\begin{equation}\label{boundedosc}
 \osc\limits_{M_k} h := \sup\limits_{M_k} h - \inf\limits_{M_k} h \leq M\gamma^k.
\end{equation}
To shorten notation, we set $a_i = \inf\limits_{M_i} h$ and $b_i = \sup\limits_{M_i} h.$
Assertion \eqref{boundedosc} will be will be proved by induction.
Let $k\in\N$ be arbitrary but fixed. We suppose $b_i-a_i\leq M\gamma^i$ for all $i\leq k$; then we need to show
\begin{equation}\label{induction}
b_{k+1}-a_{k+1}\leq M\gamma^{k+1}.
\end{equation}
By definition $M_{k+1}\subset M_k$ and therefore in particular $a_k\leq h \leq b_k$ on $M_{k+1}$. 
Set 
\[ A'=\{ z\in M_{k+1} : h(z)\leq (a_k+b_k)/2 \}.\]
Without loss of generality, assume $ 2|A'|\geq |M_{k+1}|. $ 
If this assumption does not hold, we consider $M-h$ instead of $h$. 
Let $A\subset A'$ be compact such that $3|A|\geq |M_{k+1}|.$ 
By \eqref{enter-exit1} there exists a $c_3>0$ such that $\mathbb{P}^y (T_A<\tau_k)\geq c_3$ for all $y\in M_{k+1}$.  

Let $\varepsilon>0$ and $y,z\in M_{k+1}$ such that $h(y)\geq b_{k+1}-\varepsilon$ and $h(z)\leq a_{k+1}+\epsilon$.\\
Since $h$ is harmonic, $h(X_t)$ is a martingale. We get by optimal stopping
\begin{align*}
 h(y)-h(z) = & \mathbb{E}^y \left( h(X_{T_A})-h(z) ; \tau_k>T_A \right) \\
 & + \mathbb{E}^y \left( h(X_{\tau_k}) - h(z) ; \tau_k < T_A, X_{\tau_k}\in M_{k-1} \right) \\
 & + \sum_{i=1}^\infty \mathbb{E}^y \left( h(X_{\tau_k}) - h(z) ; \tau_k < T_A, X_{\tau_k}\in M_{k-i-1}\setminus M_{k-i} \right).
\end{align*}
We will now study these three components on the right hand side seperately. 
Note $h(z)\geq a_k \geq a_{k-i-1}$ for all $i\in\N$
\begin{enumerate}
 \item In the first component, $X$ enters $A\subset A'$ before leaving $M_k$. Hence
\begin{align*}
 \mathbb{E}^y & ( h(X_{T_A})-h(z) ; T_A<\tau_k) \leq  \mathbb{E}^y \left( \frac{a_k+b_k}{2} - a_k; T_A<\tau_k \right) \\
 & = \frac{b_k-a_k}{2}\mathbb{P}^y (T_A<\tau_k) \leq \frac{M\gamma^k}{2} \mathbb{P}^y (T_A<\tau_k).
\end{align*}
\item In the component $X$ leaves $M_k$ before entering $A$. 
While leaving $M_k$, $X$ does not make a big jump in the following sense: $X$ is at time $\tau_k$ in $M_{k-1}$. 
Hence in this case $h(X_{\tau_k})\leq b_{k-1}.$ This yields to
\begin{align*}
 \mathbb{E}^y & \left( h(X_{\tau_k}) - h(z) ; \tau_k < T_A, X_{\tau_k}\in M_{k-1} \right) \leq \mathbb{E}^y \left( b_{k-1}-a_{k-1} ; \tau_k < T_A, X_{\tau_k}\in M_{k-1} \right) \\
& = (b_{k-1}-a_{k-1}) \mathbb{P}^y(\tau_k < T_A) \leq M \gamma^{k-1} (1-\mathbb{P}^y(T_A<\tau_k)). 
\end{align*}
\item In the third component $X_{\tau_k}\in M_{k-i-1}$ for $i\in\N$. Therefore
$h(\tau_k)\leq b_{k-i-1}$.
\begin{align*}
 \sum_{i=1}^\infty &\mathbb{E}^y  \left( h(X_{\tau_k}) - h(z) ; \tau_k < T_A, X_{\tau_k}\in M_{k-i-1}\setminus M_{k-i} \right) \\
 &\leq \sum_{i=1}^\infty \mathbb{E}^y \left( (b_{k-i-1}-a_{k-i-1}) ; \tau_k < T_A, X_{\tau_k}\in M_{k-i-1}\setminus M_{k-i} \right) \\
 & = \sum_{i=1}^\infty (b_{k-i-1}-a_{k-i-1})\mathbb{P}^y(X_{\tau_k}\notin M_{k-i}) 
  \leq \sum_{i=1}^\infty M \gamma^{k-i-1} c_4 \left(\frac{\rho^k}{\rho^{k-i}}\right)^{\alpha_{\max}} \\ 
 & = c_4 M \gamma^{k-1}\sum_{i=1}^\infty \left( \frac{\rho^{\alpha_{\max}}}{\gamma}\right)^i 
  = c_4 M \gamma^{k-1}\frac{\rho^{\alpha_{\max}}}{\gamma-\rho^{\alpha_{\max}}} \\
 & \leq 2c_4 \gamma^{k-2} M \rho^{\alpha_{\max}} 
  \leq \frac{c_3M\gamma^k}{2}, 
 \end{align*}
where we used 
\[ \frac{\rho^{\alpha_{\max}}}{\gamma} \leq \frac{1}{2} \quad \text{and} \quad 
 \rho \leq \left( \frac{\gamma}{2} \right)^{1/\alpha_{\max}} \Leftrightarrow \frac{\rho^{\alpha_{\max}}}{\gamma-\rho^{\alpha_{\max}}} \leq 2 \frac{\rho^{\alpha_{\max}}}{\gamma}.\]
 \end{enumerate}
Note, that the choice of $\gamma=1-c_3$ implies $ \frac{\gamma-2}{2\gamma}c_3 + \frac{1}{\gamma}+\frac{c_3}{2}=\gamma. $
Hence
\begin{align*}
 h(y)-h(z) & \leq \frac{M\gamma^k}{2} \mathbb{P}^y (T_A<\tau_k) + \frac{M\gamma^{k-1}}{r^\beta} (1-\mathbb{P}^y (T_A<\tau_k)) + \frac{c_3M\gamma^k}{2} \\
 & = M\gamma^{k} \left((\frac{\gamma-2}{2\gamma} \mathbb{P}^y (T_A<\tau_k)+\frac{1}{\gamma}+\frac{c_3}{2}\right) \\
 & \leq M\gamma^{k} \left( \frac{\gamma-2}{2\gamma} c_3 +\frac{1}{\gamma}+\frac{c_3}{2} \right) \\
 & = M\gamma^{k+1}.
\end{align*}
We conclude that
\[ b_{k+1}-a_{k+1}\leq M\gamma^{k+1}+2\varepsilon. \]
Since $\varepsilon$ is arbitrary, this proves \eqref{induction} and therefore  \eqref{boundedosc}.\\
Let $x,y\in M_1(x_0)$ and choose $k\in\N_0$ such that $M^2_{\rho^k}(x)$ is the smallest rectangle with 
$y\in M_{\rho^k}(x)$.\newline
Then $|x-y|\leq 2\sqrt{d}\rho^{k}$ and therefore  
\[k\geq \frac{\log \left(\frac{|x-y|}{2\sqrt{d}}\right)}{\log(\rho)}\quad \text{ for }  y\in M_{\rho^k}(x).\]
Hence 
\begin{align*}
 |h(y)-h(x)|& \leq M\gamma^k = Me^{k \log (\gamma)} \\
 & \leq M e^{(\log[|x-y|/(2\sqrt{d}))((\alpha_{\max}\log(\gamma))/(\alpha\log(\rho)))} \\
 &= \frac{M |x-y|^{\log(\gamma)/\log(\rho)}}{2\sqrt{d}} = c_5 M |x-y|^\beta.
\end{align*} 
Now let $h$ be harmonic on $M_r^2(x_0)$. Then $h'(x):=h(r^{\alpha_{\max}/\alpha_{\min}}x)$ is harmonic on $M_1^2(x_0)$. 
Let $x,y\in M_1(x_0)$ and $x',y'\in M_r(x_0)$ such that 
\[x=(x'_1/r^{\alpha_{\max}/\alpha_1},\dots,x'_d/r^{\alpha_{\max}/\alpha_d}), \quad  y=(y'_1/r^{\alpha_{\max}/\alpha_1},\dots,y'_d/r^{\alpha_{\max}/\alpha_d})\in M_1(x_0).\]
Then $|x-y| \leq r^{\alpha_{\max}/\alpha_{\min}}|x'-y'|.$
Set $\widetilde{x}=x'/r^{\alpha_{\max}/\alpha_{\min}}$ and $\widetilde{y}=y'/r^{\alpha_{\max}/\alpha_{\min}}$. 
We conclude
\begin{align*}
 |h(x')-h(y')| & = |h(r^{\alpha_{\max}/\alpha_{\min}}\widetilde{x})-h(r^{\alpha_{\max}/\alpha_{\min}}\widetilde{y})|= |h'(\widetilde{x})-h'(\widetilde{y})| \\ 
 &\leq c_1 |\widetilde{x}-\widetilde{y}|^\beta \sup\limits_{\widetilde{z}\in\R^d}|h(\widetilde{z})|\\
 & = c_1 \left|\frac{x'}{r^{\alpha_{\max}/\alpha_{\min}}}-\frac{y'}{r^{\alpha_{\max}/\alpha_{\min}}}\right|^\beta \sup\limits_{z\in\R^d}|h(z)| = c_1 \left(\frac{|x'-y'|}{r^{\alpha_{\max}/\alpha_{\min}}}\right)^\beta \sup\limits_{z\in\R^d}|h(z)|.
\end{align*}
\end{proof}

\end{document}